\documentclass[12pt, reqno,draft]{amsart}
\usepackage{amsfonts}

\usepackage[dvips]{graphicx}
\usepackage{latexsym}
\usepackage{amsmath}
\usepackage{amssymb,color}
\usepackage{amsthm}
\usepackage{url}
\usepackage{paralist}

\newcommand{\R}{{\mathbb R}}

\newcommand{\F}{{\mathcal F}}
\newcommand{\Q}{{\mathcal Q}}

\newcommand{\T}{{\mathcal T}}
\newcommand{\al}{\alpha}
\newcommand{\be}{\beta}

\newcounter{tmpc}

\newtheoremstyle{boldheadspec}
  {}
  {}
  {\itshape}
  {}
  {\bfseries}
  {.}
  { }
  {\thmname{#1}\thmnumber{ #2}\thmnote{ (#3)}}

\theoremstyle{boldheadspec}
\newtheorem{prop}{Proposition}
\newtheorem{theorem}{Theorem}
\newtheorem{corollary}{Corollary}
\newtheorem{lemma}{Lemma}
\newtheorem{obs}{Observation}

\theoremstyle{definition}
\newtheorem*{defn}{Definition}

\begin{document}

\title{Extremal results on  intersection graphs of boxes in $\R^d$}

\author{A.\ Mart\'{\i}nez-P\'{e}rez, L.\ Montejano and D.\ Oliveros}

\dedicatory{Dedicated to Tudor Zamfirescu.}
\maketitle
\begin {abstract}
The main purpose of this paper is to study extremal results on the intersection graphs of boxes in  $\R^d$. We calculate exactly the maximal
number of intersecting pairs in a family $\F$ of $n$ boxes in 
$\R^d$ with the property that no $k+1$ boxes in $\F$ have a point in common. This allows us to improve the known bounds for the fractional Helly theorem for boxes. We also use the Fox--Gromov--Lafforgue--Naor--Pach results to derive a  fractional Erd\H{o}s--Stone theorem  for semi-algebraic graphs in order to obtain a second proof of the fractional Helly theorem for boxes.
\end{abstract}

\section{Introduction and Results}

In \cite{B}, the authors studied the fractional behavior of the intersection structure of finite 
families of \emph{axis-parallel boxes}, or \emph{boxes} for short, in $\R^d$. 
Their aim was to prove the following statement similar to the Fractional Helly Theorem 
\cite{kathliu}: ``Let $\F$ be a family of $n$ axis-parallel boxes in $\R^d$ 
and $\al\in (1-1/d,1]$ a real number. There exists a real number $\be(\al)>0$ such that 
if there are $\al{n\choose 2}$ intersecting
pairs in $\F$, then $\F$ contains an intersecting
subfamily of size $\be n$.'' A simple example shows that 
this statement is best possible in the sense
that if $\al\leq 1-1/d$, 
there may be no point in $\R^d$ that belongs to more than $d$ elements of $\F$. 

A key idea for tackling this problem is the following notion: let $n\geq k\geq d$ and let $T(n,k,d)$ denote the maximal
number of intersecting pairs in a family $\F$ of $n$ boxes in 
$\R^d$ with the property that no $k+1$ boxes in $\F$ have a point in common.
The following bound was also obtained in \cite{B}:
\begin{equation}
T(n,k,d)<\frac{d-1}{2d}n^{2}+\frac{2k+d}{2d}n.
\end{equation}

It is not difficult to determine $T(n,k,d)$ precisely when $d=1$:
\begin{equation}
T(n,k,1)={n\choose 2}-{n-k+1\choose 2}.
\end{equation}

In fact, the graph with $n$ vertices that is the complement of  the complete graph with $n-k+1$ vertices 
is the extremal graph which is  the intersection graph of a collection of $k-1$ copies of one interval and
$n-k+1$ disjoint intervals on top, and  has 
$T(n,k,1)={n\choose 2}-{n-k+1\choose 2}$ edges.

The purpose of this paper is to determine $T(n,k,d)$ precisely, and show the following theorem:

\begin{theorem} 
For every $n\geq k>d \geq 1$,
\[T(n,k,d)=t(n-k+d,d)+T(n,k-d+1,1),\]
where $t(n,m)$ denotes the number of edges of the Tur\'an graph $\T(n,m)$.
\end{theorem}

For every $n\geq k>d \geq 1$, we shall describe a family  $\F$ of  $n$ boxes in $\R^d$ with the property that no $k+1$ boxes in  $\F$ have a point in common but the number of intersecting pairs is $t(n-k+d,d)+T(n,k-d+1,1)$. In fact, we shall precisely describe an intersection graph of this family which is an extremal graph of this problem.

\begin{corollary} 
For every $n\geq k>d \geq 1$,
\[
T(n,k,d) \leq \frac{d-1}{2d}n^{2}+(\frac{k}{d}-1)n +\frac{k}{2}(1-\frac{k}{d})<\frac{d-1}{2d}n^{2}+\frac{2k+d}{2d}n,
\]
and
\[
\left| T(n,k,d) - (\frac{d-1}{2d})n^{2}+(\frac{k}{d}-1)n) \right| 
\]
\noindent as function of $n$ is bounded by a constant that only depends on $k$ and $d$.

\end{corollary}

This corollary allows us to obtain the best Helly Fractional Theorem for boxes, although if no importance is given to the constants, 
a very interesting approach, using the work of Fox--Gromov--Lafforgue--Naor--Pach \cite{FGL} for semi-algebraic 
graphs \cite{FP}, is given in Section \ref{sec:semi-algebraic}.

\section{Technical  Propositions}

In this section we will give some definitions and basic technical propositions.

For two given integers $n\geq m\geq 1$, the Tur\'an graph $\T(n,m)$ is a complete
$m$-partite graph on $n$ vertices in which the cardinalities of the 
$m$ vertex classes are as close to each other as possible.  Let 
$t(n,m)$ denote the number of edges of the Tur\'an graph $\T(n,m)$.
It is known that $t(n,m)\leq (1-\frac{1}{m})\frac{n^2}{2}$, and
equality holds if $m$ divides $n$. In fact,
\begin{equation}\label{turangraph}
\lim_{n\to\infty}\frac{t(n,m)}{\frac{n^2}{2}}=1-\frac{1}{m}. 
\end{equation}
For completeness define $t(n,1)=0$.
For more information on the properties of Tur\'an graphs see, for example,
the book of Diestel \cite{D}.

\begin{lemma}\label{quita-r}
 For $1 \leq r \leq d$, \[ t(d,r)-r \leq t(d,d)-d. \]
\end{lemma}
\begin{proof} If $d/2 \leq r \leq d$, then the Tur\'an graph $\T(d,r)$ is the complement of the graph with $d$ vertices and  $d-r$ pairwise non-intersecting edges. So, since $t(d,d)={d \choose 2}$, we have that $t(d,d)-t(d,r)=d-r$.

If $1 \leq r \leq d/2$, then 
\[t(d,r)-r  \leq (1-\frac{1}{r})\frac{d^2}{2}-r \leq (1-\frac{1}{d})\frac{d^2}{2}-d. \qedhere \]
\end{proof}

\begin{lemma}\label{restat} For $1 \leq d \leq n$, 
\[ t(n+d,d) - t(n,d) = (d-1)n +  {d \choose 2}. \]
\end{lemma}

\begin{proof} Simply note that the complete $d$-partite graph  $\T(n+d,d)$ is obtained from the complete $d$-partite graph $\T(n,d)$ by adding one vertex to every vertex class.
\end{proof}

Furthermore,
 
\begin {obs} \label{restaT}For $n\geq k$ and $d \geq 1$,
\[ T(n+d,k,1) - T(n,k,1) = d(k-1).\]
\end{obs}

To simplify the notation let us define, for $n\geq k > d \geq 1$,
\begin{equation}
\Psi (n,k,d)=t(n-k+d,d)+T(n,k-d+1,1).
\end{equation}

\begin{obs}\label{restaphi} For every $n\geq k > d \geq 1$, 
\[ \Psi (n+d,k,d) - \Psi (n,k,d) = (d-1)n +k + {d \choose 2} -d.\]
\end{obs}
\begin{proof} $\Psi (n+d,k,d) - \Psi (n,k,d)= t((n+d-k)+d,d)-t(n-k+d,d))+ T(n+d,k-d+1,1)-T(n,k-d+1,1)$. So, by Lemma \ref{restat} and 
Observation~\ref{restaT}, $\Psi (n+d,k,d) - \Psi (n,k,d)=(d-1)(n-k+d)+ {d \choose 2} + d(k-d) =  (d-1)n +k + {d \choose 2} -d$. 
\end{proof}

\begin{prop}\label{psiigual} For every $ k > d \geq 1$ and $ 0\leq s < d$,
\[\Psi (k+s,k,d) = t(k+s,k)={k+s \choose 2}-s.\]
\end{prop}

\begin{proof} By definition, $\Psi (k+s,k,d)=t(s+d,d)+T(k+s,k-d+1,1)= t(s+d,d)+{k+s\choose 2}-{s+d\choose 2}$, but $t(s+d,d)$ is the number of edges of the Tur\'an  graph $\T(s+d,d)$ which is the complement of the graph with $s+d$ vertices and  $s$ pairwise non-intersecting edges (since $0 \leq s \leq d$). So, $t(s+d,d)= {s+d\choose 2} -s$.  Thus  $\Psi (k+s,k,d) ={k+s \choose 2}-s= t(k+s,k)$. \end{proof}

\section{The Extremal Result}

In this section we will prove our main theorem. We start by proving the following proposition.

\begin{prop}\label{menorque} For $n\geq k > d \geq 1$, 
\[T(n,k,d) \leq \Psi (n,k,d).\]
\end{prop}

\begin{proof} 
The proof is by induction on $n$. If $n=k+s$ with $0\leq s<d$, then it is clear that  $T(n,k,d) \leq t(n,k)$. Then by Proposition
\ref{psiigual}, $t(n,k)=\Psi (n,k,d)$ and we have that our proposition is true when $k\leq n < k+d$.

 Suppose the proposition is true for $n\geq k$. We shall prove it for $n+d$.  
Let $\F$ be a family of  $n+d$ boxes in 
$\R^d$ with the property that no $k+1$ boxes in $\F$ have a point in common, $n\geq k$ and $d \geq 1$. 
Let $G_{\F}$ be the intersection graph of $\F$. 

We shall prove that $\left| E(G_{\F})\right| \leq \Psi (n+d,k,d)$.\\

Let $B\in \F$. Then $B$ is of the form 
$B=((a_{1}(B),b_{1}(B))\times \dots \times (a_{d}(B),b_{d}(B)))$. We may assume
by standard arguments that all numbers $(a_{i}(B),b_{i}(B))$ ($B\in \F$)
are distinct. 

Next we will define $d$ distinct boxes $B_1,\ldots,B_d \in \F$ in the following way. Set 
\[c_1=\min \{b_1(B)\colon B \in \F\}\]
and define $B_1$ via $c_1=b_1(B_1)$. The box 
$B_1$ is uniquely determined, as all $b_1(B)$ are distinct numbers. Assume now
that $i<d$ and that the numbers $c_1,\ldots,c_{i-1}$, and boxes 
$B_1,\ldots,B_{i-1}$ have been defined. Set 
\[c_i=\min \{b_i(B)\colon B \in \F
\setminus\{B_1,\ldots,B_{i-1}\}\}\]
and define $B_i$ via $c_i=b_i(B_i)$ which, again,
is unique.

We partition $\F$ into three parts. First, let $\F_0 =\{ B_1, \dots, B_d\}$; 
second, let  $\F_1$ be the set of all boxes of $\F\setminus \F_0$ 
that intersect every $B_i$. Third, let $\F_2=\F\setminus(\F_0\cup \F_1)$.

First note that the intersection graph $\langle \F_1 \rangle$ (the generated subgraph of $\F_1$) is a complete 
subgraph of $G_{\F}$ because every box of  $\F_1$ contains the point $(c_1,\dots ,c_d)\in \R^d$. 
 
Let $S=\{e\in E(G_{\F})\mid e=\{ x,y\},  \: x\in \F_0\}$. 
 We shall prove that $\left| S\right| \leq (d-1)n +k + {d \choose 2} -d$. For this purpose, observe that 
$S= E(\langle  \F_0 \rangle)\cup E( \F_0,  \F_1) \cup E( \F_0,  \F_2)$.  

\begin{enumerate}[1)]
\item $\left| E( \F_0,  \F_1) \right| \leq d\left| V(\langle F_1 \rangle)\right| $, because $\left| V(\langle F_0 \rangle)\right|=d$.
\item $\left| E( \F_0,  \F_2) \right| \leq (d-1)\left| V(\langle F_2 \rangle)\right| $, because a point $v\in\F_2$ can not be adjacent to every point of  $ \F_0$.
\setcounter{tmpc}{\theenumi}
\end{enumerate}
 Finally, let $r=\omega (\langle F_0\rangle)$, the clique number of $\langle F_0 \rangle$. Then
\begin{enumerate}[1)]
\setcounter{enumi}{\thetmpc}
\item $\left| E(\langle F_0 \rangle)\right|\leq t(d,r)$, by the Turan Theorem.
\end{enumerate}

Therefore 
\[\left| S\right| \leq t(d,r) + d\left| V(\langle F_1 \rangle)\right| + (d-1)\left| V(\langle F_2 \rangle)\right| . \]

Since $\left| V(\langle F_2 \rangle)\right|=n-\left| V(\langle F_1 \rangle)\right|$, we have that 
\[\left| S\right| \leq (d-1)n + (\left| V(\langle F_1 \rangle)\right|+r) + t(d,r)-r.\]

Remember that $\langle F_1 \rangle$ is a complete subgraph. Hence, since $r=\omega (\langle F_0\rangle)$ and the fact that  that no $k+1$ boxes in $\F$ have a point in common, we have that $\left| V(\langle F_1 \rangle)\right|+r \leq k$, and hence by Lemma 1 that:
\[\left| S\right| \leq (d-1)n + k + t(d,d)-d=\Psi (n+d,k,d)-\Psi(n,k,d).\]

The family  $\F\setminus \F_0$ has $n$ boxes, and no $k+1$ of them have a point in common; 
hence by induction $\left| E( \langle \F\setminus \F_0\rangle)\right|\leq \Psi (n,k,d)$. Then 
$\left| E(G_{\F})\right|\leq  \quad  \mid S \mid + \left| E(\langle \F\setminus \F_0\rangle)\right| \leq  \Psi (n+d,k,d)$ as we wish. 
\end{proof}

\medskip

\begin{prop}\label{mayorque}
For $n\geq k > d \geq 1$, 
\[\Psi(n,k,d)\leq T(n,k,d).\]
\end{prop}

\begin{proof} 
Let $\Q^d=[-1,1]^d$ be the standard $d$-dimensional cube in $\R^d$. For $1\leq i\leq d$ and $-1\leq t\leq 1$ let
$\Q_{i,t}^d\subset \Q^d$ be a $(d-1)$-dimensional box defined as follows: 
\[\Q_{i,t}^d=\{ x\in \Q^d\ | \text{ if } \ x=(x_1,x_2,\dots ,x_d), x_i=t \}.\]

Observe that $\Q_{i,t_k}^d \cap \Q_{i,t_l}^d = \emptyset$ if $t_k\not= t_l$, both parallel to each other and perpendicular to the $i$-axis. 

Next, consider integer numbers $q_1,q_2,\dots q_d$ such that $q_1+q_2+\dots +q_d=n-k+d$ 
and such that $|q_i-q_j|\leq 1$ for every $i,j\in \{1,\dots ,d\} $.

We define $\F_{q_i}:=\Q_{i,t_{i_1}}^d\cup \Q_{i,t_{i_2}}^d \cup \dots \cup \Q_{i,t_{i_{q_i}}}^d$ 
where $t_{i_k}\not= t_{i_j}$ for every $k,j \in \{1,\dots q_i\}$ and define $\F_0$ as the union of $k-d$ copies of $\Q^d$.
Then 
\[\F_1:=\F_{q_1} \cup  \F_{q_2}\dots  \cup \F_{q_d}.\]  

Observe that $\F_1$ is a family of $n-k+d$ boxes in $\R^d$ where every element in $\F_{q_i}$ intersects every element on $\F_{q_j}$
if $i\not= j$. Then the intersection graph $G_{\F_1}$ is a complete $d$-partite graph which is the Tur\'an graph $\T(n-k+d,d)$, and thus the number of 
intersecting pairs in $\F_1$ is $t(n-k+d,d)$.    
Thus 
\[\F:=\F_1\cup\F_0\]
is a family of $n$ boxes where no $k+1$ of them have a point in common.  Furthermore, every element in $\F_0$ intersects 
every element in $\F_1$, so the intersection graph 
of $\F$, $G_{\F}$  has $t(n-k+d,d)+T(n,k-d+1,1)$ edges.
\end{proof}

We are ready now for our main theorem.

\begin{theorem}[Main Theorem] \label{MainTheorem} For every $n\geq k>d \geq 1$,
\[T(n,k,d)=t(n-k+d,d)+T(n,k-d+1,1),\]
and for $n\geq k$, $d \geq 1$ and $k\leq d$,
\[T(n,k,d)=t(n,k).\]
\end{theorem}

\begin{proof}  The first part follows immediately from Propositions \ref{menorque} and \ref{mayorque}.  The second part follows from the fact that $k<d$, and that the Tur\'an graph $\T(n,k)$ is the intersection graph of a family of boxes in $\R^d$. 
\end{proof}

\section{Semi-algebraic Graphs}
\label{sec:semi-algebraic}

\begin{defn}A graph $G$ is \emph{semi-algebraic} if its vertices are represented by a set of points in $P\subset \R^d$ and its edges are defined as pairs of points $(p,q)\in P\times P$ that satisfy a Boolean combination of a fixed number of polynomial equations and inequalities in $2d$-coordinates. For example, the intersection graph of a finite family of boxes in $\R^d$ is semi-algebraic.

An \emph{equipartition} of a finite set is a partition of the set into subsets whose sizes differ by at most one.
\end{defn}

\begin{theorem} [Fox--Gromov--Lafforgue--Naor--Pach \cite{FGL}] 
\label{thm:Fox} Given $\epsilon>0$, there is $K(\epsilon)$ such that if $k\geq K(\epsilon)$, the following statement is true.
 For any $n$-vertex semi-algebraic graph $G$,  there is an equipartiton of the set of vertices $V(G)$  into $k$  classes such that, with the exception of  at most a fraction $\epsilon$  of all pairs of classes, any two classes are either completely connected in $G$ or no edge of $G$ runs between them. 
\end{theorem}


As a corollary we obtain a ``fractional Erd\H{o}s--Stone theorem" (see \cite{D}) for the family of semi-algebraic graphs. That is;


\begin{theorem} \label{thm:semi-algebraic}
Given $\epsilon >0$, there is $\beta(\epsilon)>0$ such that if  $G$ is a semi-algebraic graph with $n$ vertices and more than $(1-\frac{1}{d}+\epsilon)\frac{n^2}{2}$ edges,  $G$ contains a complete $(d+1)$-partite subgraph, with each class being almost the same size  $\beta(\epsilon) n$ (a Tur\'an graph).
\end{theorem}

\begin{proof}  
Construct a ``super-graph" $G^\prime$ whose vertices are the $k$ classes of the partition given by Theorem \ref{thm:Fox}, two classes being joined by an edge of $G^{\prime}$ if all possible edges between them belong to $G$. Note that by our assumption, if $k$ is  big enough,  we can apply Tur\'an's theorem to the graph $G^\prime$  to conclude that it contains a complete graph of $d+1$ vertices. This means that there is a complete $(d+1)$-partite subgraph of $G$, with each class being almost the same size  $\frac{1}{k} n$. 
\end{proof}

As an immediate consequence of Theorem \ref{thm:semi-algebraic} we have the ``fractional Helly theorem" for boxes.

\begin{corollary} Let $\F$ be a family of $n$ axis-parallel boxes in $\R^d$ 
and $\al\in (1-1/d,1]$ a real number. There exists a real number $\be(\al)>0$ such that 
if there are $\al \frac{n^2}{2}$ intersecting
pairs in $\F$, then $\F$ contains an intersecting
subfamily of size $\be n$.
\end{corollary}

\begin{proof}  Let $K_{2,d}$ be the complete $d$-partite graph with two vertices in each color class. The corollary follows immediately from the following well known property \cite{FP}: if $G$ is the intersection graph of a family of boxes in $\R^d$, then $G$ does not contain an  induced $K_{2,d+1} $. 
\end{proof} 

\smallskip
We thank Janos Pach for drawing this new approach to our attention.

\section{Acknowledgements}
The second and third author wish to acknowledge support by CONACyT under
project 166306, and the support of PAPIIT under project IN112614 and IN101912 respectively.
The first author was partially supported by MTM 2012-30719.

\end{document}